\documentclass[12pt]{amsart}
\usepackage{amsfonts}
\usepackage{amsmath}
\usepackage{amsxtra}
\usepackage{amssymb,latexsym}
\usepackage[mathcal]{eucal}
\usepackage{amscd}
\usepackage[pdftex,bookmarks,colorlinks,breaklinks]{hyperref}
\input xy
\xyoption{all}

\newcommand{\ch}{\mathbf{1}}

\newcommand{\Z}{\mathbb{Z}}

\newcommand{\R}{\mathbb{R}}

\newcommand{\N}{\mathbb{N}}

\newcommand{\Xb}{\mathbf{X}}

\newcommand{\al}{\alpha}

\newcommand{\ga}{\gamma}
\newcommand{\del}{\delta}

\newcommand{\ep}{\epsilon}

\newcommand{\la}{\lambda}

\newcommand{\br}{\vspace{3 mm}}

\newcommand{\tri}{\bigtriangleup}

\newcommand{\supp}{{\rm{supp\,}}}

\newtheorem{thm}{Theorem}[section]
\newtheorem{cor}[thm]{Corollary}
\newtheorem{lem}[thm]{Lemma}

\theoremstyle{definition}
\newtheorem{defn}[thm]{Definition}

\newtheorem{prob}[thm]{Problem}

\numberwithin{equation}{section}

\begin{document}
\title[A universal hypercyclic representation]
{A universal hypercyclic representation}

\author{Eli Glasner and Benjamin Weiss}

\address{Department of Mathematics\\
     Tel Aviv University\\
         Tel Aviv\\
         Israel}
\email{glasner@math.tau.ac.il}

\address {Institute of Mathematics\\
 Hebrew University of Jerusalem\\
Jerusalem\\
 Israel}
\email{weiss@math.huji.ac.il}

\date{January 22, 2013}

\begin{abstract}
For any countable group, and also for any locally compact second countable,
compactly generated topological group, $G$, we show the existence of a ``universal"
hypercyclic (i.e. topologically transitive) representation on a Hilbert space, in the sense that
it simultaneously models every possible ergodic probability measure preserving free action of $G$.
\end{abstract}



\thanks{This research was supported by a grant of ISF 668/13}

\maketitle

\tableofcontents

\section{Introduction}

A bounded linear operator $S$ on a Banach space $B$ is said to be {\em hypercyclic} if there are vectors
$v \in B$ such that the sequence $\{S^n v\}_{n \ge 0}$ is dense in $B$. It is called 
{\em frequently hypercyclic} if there are vectors $v$ such that for any non-empty open set 
$U \subset B$ one has
\begin{equation}\label{fh}
{\liminf}_{N \to \infty} \frac1N \sum_{n=1}^N \ch_U(S^nv) > 0.
\end{equation}
One way to get such a property is for there to be a globally supported $S$-invariant probability measure
$\mu$ on $B$ such that the dynamical system $(B,S,\mu)$ is ergodic. In that case (\ref{fh}) will hold for
$\mu$-a.e. $v$ by Birkhoff's ergodic theorem. We shall call $S$ {\em universal} if
for every ergodic probability measure preserving dynamical system
$\Xb = (X,\mathcal{B},\mu, T)$,
there exists an $S$-invariant probability measure $\nu$ on $B$
which is positive on every nonempty open subset of $B$ and such that the dynamical systems
$\Xb$ and $(B,Borel(B),\nu, S)$ are isomorphic.
More generally, we have the following definition.

\begin{defn}
For a topological group $G$, a linear representation as operators on a Banach 
space $B$ will be called {\em universal} if
for every ergodic probability measure preserving free $G$-action
$
\Xb =   (X,\mathcal{B},\mu, \{T_g\}_{g \in G}),
$ 
there exists an $\{S_g\}_{g \in G}$-invariant probability
measure $\nu$ on $B$ which is positive on every nonempty open subset of $B$ and such
that the $G$-actions $\Xb$ and $(B, Borel(B), \nu, \{S_g\}_{g \in G})$ are isomorphic.
\end{defn}

In this work we show the existence of a universal representation on Hilbert space 
for the following classes of groups.
\begin{enumerate}
\item
All countable discrete groups.
\item
All locally compact, second countable, compactly generated groups.
\item
Groups $G$ of the form $G = \cup_{n=1}^\infty K_n$ where $K_1 < K_2 < \cdots $ is an increasing
sequence of compact open subgroups.
\end{enumerate}


The precise statement for a countable infinite group is as follows:

\begin{thm}\label{main}
Let $G$ be a countable infinite discrete group. There exists a
faithful representation of $G$, $g \mapsto S_g$, as a group of bounded
linear operators on a separable Hilbert space $H$ 
which is universal.
\end{thm}

Note that when a transformation $T_g$ acts ergodically on $(X,\mathcal{B},\mu)$ it follows
that the operator $S_g$ is a frequently hypercyclic operator \cite[Proposition 6.23]{BM}.
In particular by choosing the $G$-system
$\Xb$ to be mixing (which is always possible), we see that each $S_g$ with $\{g^n : n \in \Z\}$ infinite,
is frequently hypercyclic. In the special case where $G=\Z$, the group of integers, we obtain
(as $S_1$) a universal hypercyclic operator on Hilbert space.

We note that N. S. Feldman has previously proved the following topological version of our result for
a hypercyclic operator (\cite{Fe} and \cite[Theorem 2.27]{BM}).

\begin{thm}
There exists a hypercyclic operator $T$ acting on a separable Hilbert space $\mathcal{H}$ which
has the following property. For any compact metrizable space $K$ and any continuous map $f : K \to K$, there exists a $T$-invariant compact set $L \subset \mathcal{H}$ such that $f$ and
$T\restriction_L$ are topologically conjugate.
\end{thm}

In this theorem the operator is not invertible, it is in fact one of the simplest hypercyclic operators
defined as follows. Let $\mathcal{H}$ be the Hilbert space of sequences $u = (u_n)_1^\infty$, 
where $u_n \in l_2(\N)$ and $\|u\|^2 = \sum_1^\infty \|u_n\|^2$.
The operator $T$ is given by $(Tu)_n = 2u_{n+1}$.
Since $l_2(\N)$ contains a copy of the Hilbert cube and any compact metric space $K$
imbeds in the Hilbert cube, there is a homeomorphism $\phi_0 : K \to l_2(\N)$.
The map $\phi : K \to \mathcal{H}$ is defined by 
$$
\phi(z)_n = \frac{1}{2^n}\phi_0(f^nz),
$$
and it is easily seen that $\phi$ intertwines the actions of $f$ and $T$.
Our operator will also be defined by first defining a function and then evaluating the function along orbits,
but both the construction of the Hilbert space and the function are considerably more involved.

In Section 2 we handle the case where $G$ is finitely generated and then, in
Section 3, use an embedding theorem of Higman, Neumann and Neumann \cite{HNN}
to prove the theorem for general countable groups.
In Section 4 we consider an analogous theorem for compactly generated, second countable 
locally compact groups. 

Our doubts as to the existence of a topological analogue of the Higman-Neumann-Neumann
theorem, which will ensure the possibility of embedding a general second countable
locally compact topological group in a compactly generated group, were recently
confirmed in \cite{CC}. Thus, if an extension of Theorem \ref{CG} to the general
second countable locally compact group is valid,  a different route needs to be found.

Finally, in the last section we will show that the theorem still holds for groups
of the form $G = \cup_{n=1}^\infty K_n$ where $K_1 < K_2 < \cdots $ is an increasing
sequence of compact open subgroups.

In the proof of Theorem \ref{main} it is not hard to see that our arguments work for any $\ell_p(G,w)$
for $1 < p < \infty$ (though not for $\ell_1(G,w)$ as we need the convergence a.e. supplied by Theorem 
\ref{JRT}). 
We conclude the introduction with the following general problem.

\begin{prob}
Which separable Banach spaces (or more generally separable Fr\'echet spaces)
can serve as universal models for ergodic, countable group actions, replacing
Hilbert space in Theorem \ref{main} ?
\end{prob}

\br

\section{The finitely generated case}\label{S:fg}

\subsection{The Hilbert space representation}
Let $G$ be a finitely generated group with $\mathfrak{a}= \{a_1^{\pm 1}, a_2^{\pm 1}, \dots, a_d^{\pm 1}\}$,
a symmetric set of generators. Set
\begin{equation}\label{rho}
\rho = \frac{1}{2d +1} \left(\del_e + \sum_{i=1}^d \del_{a_i^{\pm 1}} \right),
\end{equation}
a probability measure on $G$.
Let $p_1, p_2, \dots $ be a sequence of positive real numbers with $\sum_{n=1}^\infty p_n =1$
and such that for some $C >0$, $\frac{p_n}{p_{n+1}} \le C$ for every $n$. We set
$$
w(g) = \sum_{n=1}^\infty p_n \rho^{*n}(g) \qquad g \in G.
$$
Here $\rho^{*n} = \rho * \rho * \cdots * \rho$ ($n$-times) is the $n$-th convolutional power
of $\rho$.

For $\xi \in \R^G$ let
\begin{equation}\label{norm}
\|\xi\|^2 : = \sum_{g \in G} \xi^2(g) w(g)
=  \sum_{n=1}^\infty  \sum_{g \in G}   p_n \xi^2(g) \rho^{*n}(g).
\end{equation}
The Hilbert space $H = l_2(G,w)$ is defined as:
$$
\{\xi \in \R^G : \|\xi\| < \infty\}.
$$

%
%

\begin{lem}\label{S}
For each $a \in \mathfrak{a}$ the operator
$S_a : l_2(G,w) \to l_2(G,w)$ defined by
$$
(S_a \xi)(g) =\xi(ga)
$$
is a bounded linear operator with $\|S_a\| \le \sqrt{(2d + 1)C}$ and the map $a \mapsto S_a$
defines a linear representation of $G$ in $l_2(G,w)$.
\end{lem}

\begin{proof}
We write $\sum_{g \in G}  \xi^2(g) \rho^{*n}(g)
= \mathbb{E}\xi^2(X_1X_2\cdots X_n)$, where $X_1,X_2,\dots$ is an i.i.d. sequence
of random variables with distribution $\rho$. Note that
$$
\mathbb{E}\xi^2(X_1X_2\cdots X_n a) \le (2d + 1) \, \mathbb{E}\xi^2(X_1X_2\cdots X_nX_{n+1}),
$$
hence by (\ref{norm}),
\begin{align*}
\|S_a\xi\|^2 & = \sum_{g \in G}  (S_a \xi)^2(g) w(g)\\
& = \sum_{n=1}^\infty  \sum_{g \in G}   p_n (S_a \xi)^2(g) \rho^{*n}(g)\\
& = \sum_{n=1}^\infty  p_n  \mathbb{E}\xi^2(X_1X_2\cdots X_n a)\\
& \le  (2d + 1) \sum_{n=1}^\infty  p_n \mathbb{E}\xi^2(X_1X_2\cdots X_n X_{n+1})\\
& \le  (2d + 1) \sum_{n=1}^\infty  \frac{p_n}{p_{n+1}} p_{n+1} \mathbb{E}\xi^2(X_1X_2\cdots X_n X_{n+1})\\
& \le  (2d + 1) C \sum_{n=1}^\infty  p_n \mathbb{E}\xi^2(X_1X_2\cdots X_n X_{n+1})\\
& = (2d + 1) C \|\xi\|^2.
\end{align*}
Thus $\|S_a\| \le \sqrt{(2d + 1)C}$ as required.
Finally, the way the operators $S_a$ are defined for $a \in \mathfrak{a}$ ensures that
the map $g \mapsto S_g$ is unambiguously defined and is a representation of $G$.
\end{proof}

\begin{lem}\label{bounds}
For each $b \in G$ there is a positive constant $M_b$ such that
$$
\frac{1}{M_b} w(g) \le w(gb) \le M_b w(g),
$$
for every $g \in G$.
\end{lem}

\begin{proof}
For $a \in \mathfrak{a}$ and $g_0 \in G$, applying the inequality
$$
\|S_a\xi\|^2 \le (2d + 1) C \|\xi\|^2,
$$
obtained in the proof of the previous lemma, to the element $\xi_{g_0} \in l_2(G,w)$,
which is defined by 
$$
\xi_{g_0}(g) = 
\begin{cases}
1 & g =g_0\\
0 & \text{otherwise,}
\end{cases} 
$$
we get
$$
w(g_0a^{-1}) \le (2d+1)C w(g_0).
$$
Writing $g=g_0a^{-1}$ and $M_a = (2d+1)C$ we get
$$
w(g) \le M_a w(ga).
$$
By the symmetry of $\mathfrak{a}$ we get
$$
\frac{1}{M_a} w(g) \le w(ga) \le M_a w(g), \quad \forall g \in G.
$$
Writing $b$ as a word on $\mathfrak{a}$ and iterating we get the required constant $M_b$.
\end{proof}

\subsection{The construction of the measure}

Consider an arbitrary ergodic free $G$-action on a
standard Lebesgue probability space $\Xb = (X,\mathcal{B},\mu,\{T_g\}_{g \in G})$.
Our goal in this section is to construct a $G$-invariant probability measure on $H$
which will provide a linear hypercyclic model for $\Xb$.
%
%
%
%
%

To construct this measure on $H$ we will construct a mapping from $X$ to $H$ and for this we will need 
an ergodic theorem for certain Markov operators. Such theorems were first established by 
V. Oseledec in 1965 \cite{Os}. We will use a version from a paper of Jones-Rosenblatt-Tempelman
\cite[Theorem 3.11]{JRT}. 

Given a symmetric probability measure $\eta$ on $G$ (i.e. $\eta(E^{-1}) = \eta(E)$
for every measurable $E \subset G$)
and a $G$-action $\Xb = (X,\mathcal{B},\mu,\{T_g\}_{g \in G})$,
define the operator $A_\eta : L_p(X,\mu) \to L_p(X,\mu),\ 1\le p \le \infty$ by
$$
(A_\eta f)(x) = \sum_{g \in G} f(T_gx)  \eta(g).
$$
We have
$$
(A_\eta^n f)(x) = \sum_{g \in G} f(T_gx)  \eta^{*n}(g).
$$
We will say that $A_\eta$ is a {\em random walk operator}.
For $1 \le p \le \infty$ let $P_{\mathcal{I}}$ denote the projection 
from $L_p(X,\mathcal{B},\mu)$ onto the subspace $\mathcal{I}$ of
$T$-invariant functions.
A probability measure $\eta$ on a topological group $G$ is
{\em  strictly aperiodic}
if the support of $\eta$, $S(\eta)$, is not contained in a coset of a
proper closed normal subgroup of $G$. We say that the support of $\eta$
{\em generates} $G$ if the smallest closed subgroup
containing $S(\eta)$ is all of $G$.
(In order to avoid confusion we change, in the following theorem, some of the notations 
of the original statement.)

\begin{thm}[J-R-T]\label{JRT}
Let $\eta$ be a symmetric, strictly aperiodic, regular probability measure on $G$. 
Let $\Xb = (X,\mathcal{B},\allowbreak\mu,\{T_g\}_{g \in G})$ be measure preserving $G$-action
on a probability space.
Then $A_\eta^n f$ converges in norm to $P_{\mathcal{I}} f$ for all 
$f \in L_p(X,\mathcal{B},\mu), \ 1 \le p < \infty$, and 
$A_\eta^n f(x)$ converges for a.e. $x$ if $f \in L_p(X,\mathcal{B},\mu), \ 1 < p \le  \infty$.
\end{thm}

%

We can now prove our theorem for a finitely generated $G$.

\begin{thm}\label{FG}
For every ergodic probability measure preserving free $G$-action
$\Xb = (X,\mathcal{B},\allowbreak\mu,\{T_g\}_{g \in G})$, there exists a $\{S_g\}_{g \in G}$-invariant probability measure $\nu$ on $H =l_2(G,w)$ which is positive on every nonempty open subset of $H$ and such that the $G$-actions $\Xb$ and $(H,Borel(H),\nu,\{S_g\}_{g \in G})$ are isomorphic.
\end{thm}

\begin{proof}
1.\
{\bf The Markov operator associated with $\rho$}:
With $\rho$ as in (\ref{rho}),
define the operator $A_\rho  : L_1(X,\mu) \to L_1(X,\mu)$ by
$$
(A_\rho f)(x) = \sum_{g \in G} f(T_gx)  \rho(g).
$$
Then $A_\rho$ is a random walk operator and we have
$$
(A_\rho^n f)(x) = \sum_{g \in G} f(T_gx)  \rho^{*n}(g).
$$

\br

2.\
{\bf The factor map $\phi_f$}:
Given $f \in L_4(X,\mu)$, the series
$$
 \sum_{g \in G} f^2(T_gx)w(g)
$$
converges a.e. on $X$.
In fact, by Theorem \ref{JRT}, we have
\begin{align*}
\sum_{g \in G} f^2(T_gx)w(g) & = \sum_{g \in G} f^2(T_gx) \sum_{n=1}^\infty p_n \rho^{*n}(g)\\
& = \sum_{n=1}^\infty p_n \sum_{g \in G} f^2(T_gx)  \rho^{*n}(g)\\
& =  \sum_{n=1}^\infty p_n (A_{\rho}^n f^2)(x) < \infty .
\end{align*}

Thus the map $x \mapsto \{f(T_gx)\}_{g \in G} \in \R^G$ defines a map $\phi_f : X \to H$,
\begin{equation*}
\phi_f(x) = \{f(T_gx)\}_{g \in G}.
\end{equation*}
Note that $\mu$-a.e. $f(x) = \phi_f(x)(e)$.
We denote the corresponding push forward measure by
$$
\nu_{f} = (\phi_f)_*(\mu).
$$
Moreover, we have
$$
\phi_f(T_hx) = \{f(T_gT_hx)\}_{g \in G} =  \{f(T_{gh}x)\}_{g \in G} = S_h  \{f(T_gx)\}_{g \in G} = S_h\phi_f(x),
$$
so that $\phi_f : \Xb \to (H,Borel(H),\nu,\{S_g\}_{g \in G})$ is a factor map.
Our proof will be complete when we show that there exists a function $f \in L_4(X,\mu)$ for which
(i) the map $\phi_f$ is an isomorphism of dynamical systems and (ii) the support of the measure
$\nu_f$ is all of $H$; i.e. $\nu_f(U) > 0$ for every nonempty open subset of $H$.

\br

3.\
{\bf The induction scheme for the construction of $f$}:
We first choose a sequence of measurable sets $\{A_n\}$ which are dense in the measure algebra
$(\mathcal{B},\mu)$, and such that each element repeats infinitely often.
Next we pick a sequence of balls $\{U_n\}$ which form a basis for the topology of
$H = l_2(G,w)$. More specifically, $U_n = B_{r_n}(\xi_n)$ is the open ball of radius $r_n > 0$
centered at $\xi_n$, where $\xi_n \in H_{fin}$, the dense subspace of $H$ consisting of functions
of finite support.
We will construct inductively a sequence of finite valued functions
$\{f_n\}_{n=0}^\infty$, with $f_0 = 0$, which will converge in $L_2(X,\mu)$ to the desired function $f$.
We will also define auxiliary sequences of positive reals $\ep_n$, $\ga_n$
and $\del_n$ decreasing to $0$.

We now suppose that by stage $n$ we already have the following structure:
\begin{enumerate}

\item[$1_n$]
A function $f_n$, admitting finitely many values:
to each $i \le n$ there are two finite subsets $(V^{(n)}_{i,0}, V^{(n)}_{i,1})$ of $\R$  which are
at least $\ep_i(1 +\frac1n)$ apart, and 
the range of $f_n \ =\, \cup_{i \le n} \{ V^{(n)}_{i,0} \cup V^{(n)}_{i,1}\}$.
\item[$2_n$]
There is a number $\beta_n$ so that for each $i \le n$ the sets
$(\hat{V}^{(n)}_{i,0}, \hat{V}^{(n)}_{i,1})$, which are obtained by covering each point in the sets 
$(V^{(n)}_{i,0}, V^{(n)}_{i,1})$ with intervals of length $\beta_n$, are $\ep_i(1+\frac{1}{n+1})$ apart,
and so that for each $i < n$
and $j \in \{0,1\}$, $\hat{V}^{(n)}_{i,j} \subset \hat{V}^{(n-1)}_{i,j}$.
\item[$3_n$]
For every $i \le n$ and for
all $x \in X$, $x \in A_i$ implies $f_n(x) \in V^{(n)}_{i,0}$ and $x \in X \setminus A_i$ implies
$f_n(x) \in V^{(n)}_{i,1}$, with the exception of sets of measure $ < \ga_i(1 -\frac1n)$.
\item[$4_n$]
On a set $E^{(n)}_i$ of measure $\ge \del_i(1 + \frac1n)$, $\phi_{f_n}$ takes values in $U_i$ for $i \le n$,
with $E^{(j+1)}_i \subset E^{(j)}_i$ for $j \le n -1$.
\item[$5_n$]
$$
\|f_n\|_4  = \left(\int_X |f_n|^4 \,d\mu \right)^\frac14 < 1.
$$
\end{enumerate}

\br

4.\
{\bf Step $n+1$}:

Next find a $\eta>0$ so small that if we change $f_n$
by less than $\eta$ on a set of measure at least $1 -\eta$ to obtain $f_{n+1}$, then we will
still satisfy the conditions $1_{n+1}$ to $5_{n+1}$ for $i \le n$ with $f_{n+1}$.
Thus $\eta$ has to be less than 
$\min \{\frac{1}{2n(n+1)}\ep_i, \frac{1}{2n(n+1)}\beta_i,  \frac{1}{2n(n+1)}\del_i, \frac{1}{2n(n+1)}\ga_i \}$
for all $i \le n$.

\br

(a) {\em Hitting $U_{n+1}$ with positive probability}:
Recall that $U_{n+1} = B_{r_{n+1}}(\xi_{n+1}) \subset H$ where
$\xi_{n+1} \in H_{fin}$ is supported on a ball, say $B_{N_0}$ in $G$
(this is the collection of elements in $G$
which have a representation as a word on the alphabet $\mathfrak{a}$ of length $\le N_0$).

As $G$ acts freely on $X$ we can find $N >> N_0$ (we will determine soon how large 
$N$ needs to be) and a set $E$ of positive measure such that the
sets $\{gE\}_{g \in B_{N}}$ are disjoint and
$\mu(B_{N}E) < \frac12 \eta$ (this number determines the new $\del_{n+1}$).
For a proof of this fact see e.g. \cite[Lemma 3.1]{W}.
We erase $f_n$ on $B_{N}E$ and replace it there by a function which (i) matches the function
$\xi_{n+1}$ on $B_{N_0}E$ and (ii) is zero on the set $(B_N \setminus B_{N_0})E$.
This new function we denote by $\hat f_n$. We only have to choose $N$ sufficiently large so that a.e.
the contribution of ${\hat f}_n^2(gx)$ with $g \not\in B_N$ to the $l_2$-norm of $\phi_{f_n}$ in
$H=l_2(G,w)$ is $< \frac12 r_{n+1}$ (recall that $\hat f_n$ admits only finitely many values).

%
%

\br

(b) {\em Distinguishing $A_{n+1}$ from $X \setminus A_{n+1}$}:
Let $(\hat V^{(n)}_{i,0}, \hat V^{(n)}_{i,1})$ be closed neighborhoods of the sets
$(V^{(n)}_{i,0}, V^{(n)}_{i,1})$, obtained by covering each point with an interval
of length $ < \beta_n$, which are $\ep_i(1 + \frac{1}{n+1})$ apart.
We consider a set of constancy, say $D \subset X$, of the function $\hat{f}_n$,
where it takes the value $u \in \R$. 
We split $u$ to two values $u_0$ and $u_1$, staying within the
interiors of the corresponding sets $(\hat{V}^{(n)}_{i,0}, \hat{V}^{(n)}_{i,1})$, to distinguish between
$A_{n+1}$ and $X \setminus A_{n+1}$. We do this for all the values of $\hat{f}_n$
to define 
the new sets $(V^{(n+1)}_{n+1,0}, V^{(n+1)}_{n+1,1})$ and
a new $\ep_{n+1}$ (all the points are distinct). This also defines the new sets
$(V^{(n+1)}_{i,0}, V^{(n+1)}_{i,1})$ for $i \le n$ and the function $f_{n+1}$.
Finally we choose a number $\beta_{n+1}$ so that for each $i \le n+1$ the sets
$(\hat{V}^{(n+1)}_{i,0}, \hat{V}^{(n+1)}_{i,1})$, which are obtained by covering each point in the sets 
$(V^{(n)}_{i,0}, V^{(n)}_{i,1})$ with intervals of length $\beta_{n+1}$, are $\ep_i(1+\frac{1}{n+2})$ apart,
and so that for each $i < n +1$
and $j \in \{0,1\}$, $\hat{V}^{(n+1)}_{i,j} \subset \hat{V}^{(n)}_{i,j}$.

\br

5. {\bf Conclusion}:
The way the sequence $f_n$ was constructed ensures the existence in $L_2(X,\mu)$
of the limit function $f := \lim_{n \to \infty} f_n$.
The essential range of the function $f$
lies in the countable union of Cantor sets
$$
E = \bigcup_{i=1}^\infty C_{i,0} \cup C_{i,1},
$$
where
$$
C_{i,0} =
\bigcap_{n=i}^\infty \hat V^{(n)}_{i,0}
\ \ \text{and} \ \ 
C_{i,1} =
\bigcap_{n=i}^\infty \hat V^{(n)}_{i,1},
$$
and
$$
C_{i,0} \cap C_{i,0} =\emptyset.
$$
It is now easy to check that
the properties $3_n$ imply that the map $\phi_f$ is an isomorphism.
E.g. we have, for each $i$
$$
\mu(f^{-1}(C_{i,0}) \tri A_i) < \ga_i,
$$
and, as each $A = A_i$ appears infinitely often, we conclude that, in fact
$$
f^{-1}(C_{i,0}) = A \pmod 0.
$$
Similarly the properties $4_n$ imply that $\supp(\nu_f)=H$.
This completes the proof of Theorem \ref{FG}.
\end{proof}

\section{The general countable group case}

Our goal in this section is to prove Theorem \ref{main}. We do this by applying Theoren \ref{FG}
to an ambient finitely generated group which is provided by
the following well known result of Higman, Neumann and Neumann \cite{HNN}:

\begin{thm}
Any countable group $G$ can be embedded in a group $H$ generated by two elements.
If the number of defining relations for $G$ is $n$, the number of defining relations for $H$ can
be taken to be $n$.
\end{thm}

\begin{proof}[Proof of Theorem \ref{main}]
Let $G < H$ be a HNN extension of $G$ with $H = \langle a_1^{\pm 1}, a_2^{\pm1} \rangle$.
Construct the symmetric probability measure $\{w_H(h) : h \in H\}$ as in the proof of Theorem \ref{FG}.
By Lemma \ref{bounds} there is, for every $b \in H$, a constant $M_b >  0$ with
$$
\frac1M_b w_H(h) \le w_H(hb) \le M_b w_H(h),
$$
for all $h \in H$.
We restrict $w_H$ to $G$ and renormalize to obtain a probability measure
$\rho$ on $G$. Thus, for every $g \in G$, 
$\rho(g) = \frac1K w_H(g)$, where $K = \sum_{k \in G}w_H(k)$.
We still have that for every $b \in G$
\begin{equation}\label{nrho}
\frac1M_b \rho(g) \le \rho(gb) \le M_b \rho(g), \quad \forall g \in G.
\end{equation}
Next we use the symmetric probability measure $\rho$ --- which is of course no longer finitely
supported; in fact, it is strictly positive everywhere on $G$ --- to create a new
probability measure $w = w_G$ on $G$, namely $w(g) = \sum_{n=1}^\infty p_n \rho^{*n}(g)$
for every $g \in G$.

Now in the Hilbert space $l_2(G,w)$, by an argument similar to the one used in the proof of
Lemma \ref{S}, and using (\ref{nrho}), we have for each $g_0 \in G$ and every $\xi \in l_2(G,w)$:
\begin{align*}
\|S_{g_0}\xi\|^2 & = \sum_{g \in G}  (S_{g_0} \xi)^2(g) w(g)\\
& = \sum_{n=1}^\infty \sum_{g_1,g_2,\dots,g_n} p_n \xi^2(g_1g_2\cdots g_ng_0)
\rho(g_1)\rho(g_2)\cdots \rho(g_n)\\
& \le
M_{g_0}  \sum_{n=1}^\infty  \sum_{g_1,g_2,\dots,g_n} p_n \xi^2(g_1g_2\cdots g_{n-1}(g_ng_0))
\rho(g_1)\rho(g_2)\cdots \rho(g_{n-1})\rho(g_ng_0)\\
& = M_{g_0}  \sum_{n=1}^\infty \sum_{g_1,g_2,\dots,g_n} p_n \xi^2(g_1g_2\cdots g_n)
\rho(g_1)\rho(g_2)\cdots \rho(g_n).
\end{align*}
Thus $\|S_{g_0}\| \le  M_{g_0}$ and the map $g \mapsto S_{g}$ is the required
representation of $G$ on $l_2(G,w)$. The rest of the proof follows mutatis mutandis that of
Theorem \ref{FG}.
\end{proof}

\section{The compactly generated group case}

In this section we deal with topological groups which are not
necessarily discrete. Our goal is to obtain an analogue of Theorem \ref{FG}
for a topological group which is compactly generated, locally compact
and second countable. We will concentrate on pointing out the technical
results that are needed to carry out the proof that we gave in detail for
countable groups. The first lemma concerns the group itself equipped with 
a Haar measure.  
%

\begin{lem}\label{min}
Let $G$ be a locally compact group and $K$ a compact symmetric neighborhood
of the identity. Let $\la$ be a fixed right-invariant Haar measure on $G$.
For any compact symmetric set $L \subset G$ such that $K$ lies in the interior of $L$
(for example $K^2$) if $m$ is Haar measure restricted to $L$ 
(i.e. $m = \ch_L \la$)
then the density of the measure $m*m$ restricted to $LK$ is bounded away from zero. 
\end{lem}

\begin{proof}
Since the density of $m*m$; i.e. the function $\psi$
for which $d(m*m) = \psi(g)\, d\la(g)$, is continuous, and $LK$ is compact it suffices to show that for any
point of the form $lk$ with $l \in L$ and $k \in K$ this density is positive.
Now this density is given by
$$
\psi(lk)  = \int \ch_L(l k h^{-1})\ch_L(h) d\la(h).
$$
For all $v$ in a small symmetric neighborhood of the identity $V$ we have  $vk \in L$, since $K$ is in
the interior of $L$. For $h=vk$ of this form we have $l k h^{-1} = l v^{-1}$ and $lV$ must
intersect the interior of $L$. Thus this intersection has positive Haar measure, from which it follows
that the density must be positive:
\begin{align*}
\psi(lk) & = \int \ch_L(l k h^{-1})\ch_L(h) d\la(h) \\
        & \ge \int \ch_L(l v^{-1})\ch_L(vk) d\la(vk) > 0.
\end{align*}
\end{proof}

We remark that a similar  conclusion is valid also for $\hat{m}$ which is defined by
$\hat{m}(B) = m(B^{-1})$. 

Let now $K$ be a symmetric compact neighborhood of the identity which generates $G$
(i.e. $\cup_{n \in \N} K^n = G$).
Let $L = K^2$ and define  $\rho$ by the following equation:
$$ \rho(B) = \frac{1}{2\la(L)}( \la(L \cap B) + \la(L \cap B^{-1})).$$ 

By definition $\rho = \frac{1}{2}(m + \hat{m})$ is a symmetric measure so that when it 
averages unitary operators the result will be self adjoint. 

Finally let
$$
w = \sum_{n=1}^\infty p_n \rho^{*n},
$$
where, as in Section \ref{S:fg},  $p_1, p_2, \dots $ is a sequence of positive real numbers with
$\sum_{n=1}^\infty p_n =1$ and such that for some $C >0$, $\frac{p_n}{p_{n+1}} \le C$ for every $n$.

Our Hilbert space now is $H = L_2(G,w)$ and the representation is given by
right multiplication: $S_k\xi(g) = \xi(gk)$ for $k \in K$.

We can now formulate the main result of this section:
   
\begin{thm}\label{CG}
Let $G$ be a a topological group which is compactly generated, locally compact
and second countable.
For every ergodic probability measure preserving free $G$-action
$\Xb =   (X,\mathcal{B},\mu, \{T_g\}_{g \in G})$, there exists a $\{S_g\}_{g \in G}$-invariant probability
measure $\nu$ on $H =L_2(G,w)$ which is positive on every nonempty open subset of $H$ and such
that the $G$-actions $\Xb$ and $(H,Borel(H),\nu,\{S_g\}_{g \in G})$ are isomorphic.
\end{thm}
    
The fact that the shift operators $S_g$ are bounded is shown as follows. 
From Lemma \ref{min} we deduce:

\begin{lem}\label{rho2}
There exists a constant $D > 0$ such that for every $k \in K$, $g \in G$ and $n \in \N$:
\begin{enumerate}
\item
$\rho * \del_k \le D (\rho *\rho)$,
hence also 
$$
\rho^{*n} * \del_k \le D \rho^{*(n+1)}, \quad \forall  n \ge 1.
$$
\item
$
dw(gk) \le DC\, dw(g).
$
\end{enumerate}
\end{lem}

\begin{proof}
1.\
Applying Lemma \ref{min} we know that
$u := \min \{\psi(t) :  t \in KL \}   > 0$. We then have on $Lk$,
$m*m \ge u\, m * \del_k$. In a similar fashion we will have a positive constant
$v$ such that $\hat{m}*\hat{m} \ge v\, \hat{m} * \del_k$. Thus we obtain 
$\rho * \del_k \le \frac{2}{u \wedge v} \rho*\rho$. Convolving with $\rho^{*n}$ on the left
we finally get
$$
\rho^{*n} * \del_k \le D \rho^{*(n+1)},
$$
with $D = \frac{2}{u \wedge v}$.

2.\ As in the proof of Lemma \ref{S}.
\end{proof}

\begin{cor}
For every $k \in K$ we have:
$$
\|S_k\| \le  \sqrt{DC}.
$$
\end{cor}

For the main part of the proof of the theorem we will need a continuous version of the basic property 
of free actions of countable groups that we used in the previous section. Since this 
is probably less familiar we will spell this out in detail. 
\begin{defn}
Given a $G$ probability measure preserving action
$\Xb =   (X,\mathcal{B},\mu,\allowbreak \{T_g\}_{g \in G})$, and a compact neighborhood
$Q$ of $e \in G$, we say that a measurable subset $V \subset X$ is a {\em thin base for
a $Q$-tower}, if there is a Borel measure $\theta$ on $V$ such that the action map
$\al : Q \times V \to \al(Q \times V) \subset X$, given by $\al(g,v) =T_gv$, is one-to-one
and maps $ \ch_Q \la \times \theta $ to $\ch_{\al(Q\times V)}\mu$.
\end{defn}

The following property holds for free actions of locally compact groups.
If $Q$ is an arbitrary compact symmetric neighborhood of $e \in G$ and $E \subset X$
is any measurable subset with $\mu(E) > 0$ then, as the action is free,
by \cite[page 54]{OW}, $E$ contains a subset $V \subset E$ which is a thin base for a $Q$-tower.
The construction of the function $f$ on $X$ will be carried out by defining a sequence 
of finite valued functions $f_n$ on $Q_n$-towers as in the proof of  Theorem \ref{FG}. 
The function space $L_2(G,w)$ contains a countable set of finite valued functions that is dense 
and we can use them to define a countable collection of open balls $U_n$ as before. There is an
added technical complication which arises from the fact that the thin bases of the $Q$-towers
have measure zero. This is taken care of by the basic fact that in $L_2(G,w)$ the translation by
elements of $G$ is strongly continuous, so that having defined a function using a thin tower 
for a set of positive measure we will have 
functions which are very close to it.  

Having said all of this one can basically copy the proof that we spelled out in detail and establish
the theorem.

\section{$G = \cup K_n$}
In this short section we will indicate how one obtains a proof of a version of Theorem \ref{CG} 
for a locally compact, second countable group $G$ which is not compactly generated but has a 
rather simple structure: $G = \cup_{n=1}^\infty K_n$, where $K_1 < K_2 < \cdots $ is an increasing
sequence of compact open subgroups.

\begin{thm}\label{union}
The statement of Theorem \ref{CG} holds for every topological group $G$ as above. 
\end{thm}

\begin{proof}
As is the case of Theorem \ref{CG}, we only need to
demonstrate an analogue of the key inequality of Lemma \ref{rho2}.
The role of $\rho$ is now played by the probability measure
$$
\rho = \sum_{n=1}^\infty p_n \la_n,
$$
where $\la_n$ is the normalized Haar measure of the compact subgroup $K_n$.

\begin{lem}
For every $g_0 \in G$ there exists a constant $C_{g_0}$ such that
$$
\rho * \del_{g_0} \le C_{g_0} (\rho *\rho).
$$
\end{lem}

\begin{proof}
Let $m_0$ be the first integer with $g_0 \in K_{m_0}$. Then
$$
\rho * \del_{g_0}  = \sum_{n=1}^\infty p_n \la_n * \del_{g_0} =
\sum_{n=1}^{m_0 -1} p_n \la_n * \del_{g_0}  +  \sum_{n=m_0}^\infty p_n \la_n. 
$$
Moreover, for every $n \le m_0$ we have the inequality
$$
\la_n * \del_{g_0} \le [K_{m_0} : K_n] \la_{m_0} 
\le [K_{m_0} : K_1] \la_{m_0}.
$$

On the other hand an easy computation shows that
\begin{align*}
\rho * \rho & =  \sum_{i=1}^\infty \sum_{i=1}^\infty p_i p_j \la_i * \la_j \\
                 & = \sum_{i=1}^\infty \sum_{j=1}^\infty p_i p_j \la_{i\vee j} \\
                 & \ge   \sum_{k=1}^\infty \left(\sum_{i \le k} p_i \right)  p_k \la_k \\
                 & \ge p_1    \sum_{k=1}^\infty p_k \la_k.                             
\end{align*}
Thus taking $C_{g_0} = \frac{1}{p_1} + [K_{m_0} : K_1]$ we obtain the required inequality.
\end{proof}

Again we finish the proof as in Theorem \ref{CG}.
\end{proof}

\end{document}